\documentclass[12pt,a4paper]{article}
\usepackage{amsmath}
\usepackage{amssymb}
\usepackage{graphicx}
\usepackage{color}
\usepackage{amsthm}
\usepackage{amsfonts}
\numberwithin{equation}{section}
\title{Total flooding time and rumor propagation on graphs}
\author{Darcy Camargo \and Serguei Popov}
\newcommand{\pr}{\mathbb{P}}
\newcommand{\ex}{\mathbb{E}}

\newcommand\laweq{\mathrel{\overset{\makebox[0pt]{\mbox{\tiny law}}}{=}}}

\newtheorem{teo}{Theorem}[section]
\newtheorem{lemma}[teo]{Lemma}
\newtheorem{df}[teo]{Definition}
\newtheorem{prop}[teo]{Proposition}
\newtheorem{coro}[teo]{Corollary}

\newtheorem{openp}[teo]{Open Problem}
\begin{document}
\bibliographystyle{plain}
\maketitle

{\footnotesize  \noindent Department of Statistics, Institute of Mathematics,
 Statistics and Scientific Computation, University of Campinas --
UNICAMP, rua S\'ergio Buarque de Holanda 651,
13083--859, Campinas SP, Brazil\\
\noindent e-mails: \texttt{darcygcamargo@gmail.com, popov@ime.unicamp.br}

}
\begin{abstract}
We study a model of rumor propagation in discrete time where each 
site in the graph has initially a distinct information; we are interested in the number of ``conversations'' before the entire graph knows all informations. This problem can be described as a flooding time problem with multiple liquids. For the complete graph we compare the ratio between the expected propagation time for all informations and the corresponding time for a single information, obtaining the asymptotic ratio $3/2$ between them.
\\[.3cm]\textbf{Keywords:} rumor propagation, flooding time, complete graph, first-passage percolation
\\[.3cm]\textbf{AMS 2000 subject classifications:} 60K30, 60J10.
\end{abstract}
\section{Introduction and results}

The well-known and well-studied model of First-passage percolation is a random process on a graph ${G}=(V,E)$, introduced by Hammersley and Welsh in~\cite{FPP}. It can be described in the following way: Consider a sequence of i.i.d.\ random variables $T_e$ indexed by the edges of the graph, and call~$T_e$ the passage time of edge $e$. For a path $\gamma$ we define the passage time of $\gamma$ as the sum of the passage times of the edges of $\gamma$. Then the passage time from site~$a$ to site~$b$ is the minimum of the path's passage times over all possible paths from~$a$ to~$b$. Usually one is interested in questions related to the passage time of extremal sites of the graph or geodesics with respect to the metric induced by the process.

The notion of flooding time arises from the same process, but we are interested in the time until \emph{all} sites are reached from the initial site by paths. So if we denote the passage time from~$a$ to~$b$ by $\mathcal{L}(a,b)$, then the flooding time from site~$v_0$ is $\sup_{v\in E} \mathcal{L}(v_0,v)$. The usual way to describe the flooding time is that there exists some initial site~$v_0$ with some spreading property, for example an infection source (for infection models) or a water source (justifying the name flooding time). Here we will consider a rumor model in a network scenario, where the site is spreading a particular information it knows. We will assume the passage times along edges to be exponentially distributed with parameter 1, so we can work only with the skeleton discrete process. A more detailed discussion of flooding times can be found in~\cite{flooding}.

%It is very straightforward to see that if in our graph ${G}=(V,E)$ we consider the discrete skeleton to the continuous process as described, then the propagation time (i.e. the number of conversations until all sites knows all informations in our graph) is just $|E|$ times the flooding time, since our flooding process is a Markovian process with contant transition rate $|E|$.

%The main difference between our studies here from the existing ones are that we are not interested in the flooding (or the propagation) time of one information, but in comparing it to the time until all informations are known by all sites, where each site initially knows a distinct information. 

The initial works about rumors are by Daley and Kendall in~\cite{DK} and Maki and Thompson in~\cite{MT}, these classical models are considered in a continuous time case similarly to the classical flooding time situation. In these models we usually have three kinds of individuals: ignorants (those that are yet to discover the rumor), spreaders (those that know the rumor and are spreading it) and stiflers (those that know the rumor, but do not spread it). These models and their variations are already well studied, see for example \cite{art1, art2, art3, art4, art5}. Differently from the Daley and Kendall or Maki and Thompson models, here we study the propagation of the rumor in discrete time and with no stiflers, so each individual will always propagate the information over time if possible. 

Our main question is how long does it take to propagate all informations in comparison to the corresponding time for just one information, in particular, in the complete graph with~$n$ sites, denoted by~$K_n$. Its symmetry allows us to obtain the result that the ratio of the expected time to propagate~$n$ informations to the expected time for just one information is asymptotically~$3/2$.

 Let ${G}=(V,E)$ be a graph with site set~$V$ and edge set~$E$, we refer to an edge that connects site~$x$ to site~$y$ as~$e_{xy}$. If $|V|=n$ we denote by~$I_n=\{1,2,\ldots,n\}$ the set of all $n$ informations and by $i_x(t)\subset I_n$ the set of informations known by site $x$ at time $t$. We can construct a vector of subsets of~$I_n$ that describes the configuration of the informations among the sites at time~$t$. Let us denote by~$\mathcal{I}_t=\mathcal{I}_t(V)$ this information vector, so if $V=\{x_1, x_2,\ldots,x_n\}$, then $\mathcal{I}_t=(i_{x_k}(t),\, k=1,2,\dots,n)$.

Now consider a graph ${G}=(V,E)$ with $|V|=n$ and a initial configuration of informations $\mathcal{I}_0$ to construct our discrete time rumor process in the following way.
\begin{enumerate}
\item At time $t$,  select an edge from $E$ uniformly at random, say $e_{xy}$ was selected.
\item Then the sites that are connected by the selected edge share their informations, so $i_x(t+1)=i_y(t+1)=i_x(t) \cup i_y(t)$, and all other sites stay unchanged, that is, $i_z(t+1)=i_z(t)$ for $z \neq x,y$.
\end{enumerate}

We are mostly interested in the expectation of the following two random variables: the first is the time until some specific information $k\in I_n$ becomes known to the entire graph and the second is the time until all informations of~$I_n$ are known by the all the sites of the graph. So let us define the propagation time of a set of informations $H\subset I_n$: 
\begin{equation}
\tau_H:=\inf\{t: H \subset i_x(t) \mbox{ for all } x \mbox{ in } V\},
\end{equation}
 and we also denote by $\tau_{V}:=\tau_{I_{|V|}}$ the total propagation time (i.e., the propagation time for all informations of all sites).

For convenience, we will use $\tau_k$ instead of $\tau_{\{k\}}$. On transitive graphs it holds that~$\tau_k$ and~$\tau_m$ have the same distribution for all $k,m \in I_n$.

Our first result is a bound of the total propagation time on general graphs.

\begin{teo} 
\label{teoG}
Consider a graph ${G}=(V,E)$ with the propagation process starting from the configuration where each site knows a different information, then for any $x$ in $V$ the total propagation time $\tau_V$ satisfies
\begin{equation}
\ex[\tau_V]\leq 2\ex[\tau_x].
\end{equation}
\end{teo}

This result motivates a definition about the ratio of the total propagation time and the propagation time of one information. 
\begin{df}[Propagation ratio]

Consider a propagation process graph ${G}=(V,E)$ with the  with the initial configuration where each site knows a different information. We define the \textit{Propagation ratio} of ${G}$ as value
\[
\mathcal{R}(G)=\frac{\ex\tau_V}{\min \limits_{x \in V} \ex\tau_x}.
\]
\end{df}

By definition we have $\mathcal{R}(G)\geq 1$ for any graph. A example of family of graphs that have propagation ratio asymptotically $1$ is the ring graph. Indeed, we have
\begin{prop} 
 \label{prop_ring}
Let $R_n$ be the graph with vertex set $\{1,\ldots,n\}$
and edge set $\{(1,2),\ldots,(n-1,n), (n,1)\}$.
We have $\ex\tau_V/\ex\tau_1\to 1$ as $n\to\infty$.
\end{prop}

 By Theorem \ref{teoG} we have that for any graph $G$ we have $\mathcal{R}(G)\leq2$, so a natural question arises from this:  does the total propagation ratio goes to 2 for some graph family? The following proposition answers this question.

\begin{prop} 
\label{propS}
Consider a star graph $S_n=(V_n, E_n)$,with $V_n=\{0,1,\ldots,n\}$,  $E_n=\{(0,x), x=1,2,\ldots,n\}$ and the propagation process starting from the configuration where each site knows a different information, then for any $x$ in $V$ the total propagation time $\tau_V$ satisfies
\begin{equation}
\lim \limits_{n\to\infty} \mathcal{R}(S_n)=2
\end{equation}
\end{prop}

This proposition shows us that the supremum of possible propagation ratios over all graphs is indeed $2$, but it leaves open some interesting problems.

\begin{openp} For any $\alpha \in (1,2)$, does there exist 
a family of graphs  ${G}_{\alpha,n}$ for which the propagation 
ratio converges to $\alpha$ as $n\to\infty$?
\label{openp1}
\end{openp}

\begin{openp} What is the supremum of the propagation ratios over the class of transitive graphs? (Observe that the star graph is
not transitive.)
\label{openp2}
\end{openp}

%We believe that the conjecture of Open Problem~\ref{openp1} is true. 
%%% are you really sure about that?
Also, although we do not have the exact answer for Open Problem~\ref{openp2}, the main result in this work gives a lower bound for this value.

Now, consider the complete graph~$K_n$ with $n$ sites denoted by $\{x_1,x_2,\ldots,x_n\}$.  Let us consider the following two initial scenarios for the distribution of informations. 

Consider the initial scenario $\mathcal{I}_0^1$: Each of the $n$ sites knows a different information ($i_{x_k}(0)=\{k\}$, for $k=1,2,\ldots,n$). In the complete graph, we have by the symmetry that the expected propagation time of a set of informations is a function of the number of individuals that initially know these informations (for the total propagation time this number coincides with the number of informations).

We define the quantity $M_n$ by
\begin{equation}
\label{Mn}
M_n(k)=\ex_{\mathcal{I}_0^1}\tau_{\{1,2,\ldots,k\}},\quad \mbox{for } k=1,2,\ldots,n.
\end{equation}
So, $M_n(k)$ is the expected time for a certain set of $k$ informations to be known by the entire system, given the initial scenario of each site knowing a different information.

Now consider a slightly different initial scenario $\mathcal{I}_0^2$ that will play an important role in our results: This time, for the $n$ sites of the complete graph, we have only $n-1$ informations. The first information is known by two sites and the others are known each by a different site, so $i_1(0)=i_2(0)=\{1\}$ and $i_x(0)=\{x-1\}$ for $x=3,4,\ldots,n$.

Then with the described scenario define the quantity $A_n$ by
\begin{equation}
\label{An}
A_n(k)=\ex_{\mathcal{I}_0^2}\tau_{\{1,2,\ldots,k-1\}}, \quad \mbox{for } k=2,3,\ldots,n.
\end{equation}

Before stating our main results, we recall a result about the propagation time for one information (of course, it is a quite simple fact):

\begin{prop}
\label{prop1}
On the complete graph with $n$ sites, the expected time for one specific information to propagate is
\begin{equation}
M_n(1)=(n-1)\sum \limits_{k=1}^{n-1} \frac{1}{k} \sim n\ln n.
\end{equation}
\end{prop}

Then, our main theorem is

\begin{teo}
\label{mainteo}
On the complete graph $K_n$, the expected total propagation time $M_n(n)$ satisfies
\begin{equation}
\lim \limits_{n \to \infty} \frac{M_n(n)}{M_n(1)}=\frac{3}{2}.
\end{equation}
\end{teo}

So, using this result and Proposition~\ref{prop1}, we obtain as a corollary the asymptotic behavior of $M_n(n)$:
\begin{equation}
\label{asymK} \quad M_n(n) \sim \frac{3}{2}n\ln n.
\end{equation}

As a consequence,  we obtain that being $\mathcal{T}$ the class of transitive graphs, we have 
\[
\sup \limits_{G\in \mathcal{T}} \mathcal{R}(G) \geq \frac{3}{2}.
\]

We conclude the introduction by transcribing~\eqref{asymK} into the flooding time version (i.e., in continuous time). 

\begin{coro} Let $\tau_{flood}^{({G})}$ be the maximum of all the flooding times in graph ${G}=(V,E)$ from all possible distinct sources, i.e. $\tau_{flood}^{({G})}=\max \limits_{v \in V} \tau_{flood}^v$. Then we have
\begin{align*}
\ex\left[\tau_{flood}^{(K_n)}\right]&\sim \frac{3}{n}\ln n.
\end{align*}
\end{coro}

%section2
\section{Preliminary results}
As a warm-up, we first prove the three propositions
from the previous section.

\begin{proof}[Proof of Proposition~\ref{prop_ring}]
The proof is standard, so we give only a sketch.
 This follows from the usual concentration argument:
since~$\tau_1$ is a sum of i.i.d.\ (Geometric) random variables,
it deviates from its mean by a linear amount with exponentially
small probability. Therefore, the maximum of~$n$ 
identically distributed propagation times has essentially the same
asymptotic behaviour as one propagation time.  
\end{proof}

\begin{proof}[Proof of Proposition~\ref{propS}]
From now on, we denote the harmonic sum by $H(m)=\sum_{j=1}^m j^{-1}$. We first focus on the informations reaching the central site of the star graph (site $0$). As each edge from the graph must be chosen at least once, then this is a coupon collector problem with $n$ coupons and so this expected time will be $n H(n)$. So, after all informations reach $0$, the last information will be known only by two sites and needs to be propagated by the central site. Also, all other informations will be propagated together, as the central site knows all of them. Then we again fall into a coupon collector problem, but as now initially we have one site besides the central one knowing the last information, then we already have one coupon and must collect other $n-1$. So, for any $x\neq 0$ we have:
\[
\ex\tau_V=n H(n) +(nH(n)-1)=2nH(n)-1
\]

As in the star graph the minimal propagation time is from the central site information, we have again a coupon collector problem with $\ex \tau_0 =nH(n)$ and so:

\[
\mathcal{R}(S_n)=\frac{\ex\tau_V}{\ex\tau_0}=2-\frac{1}{nH(n)}\to2,\quad \mbox{as $n\to \infty$}
\]

which concludes the proof.
\end{proof}

\begin{proof}[Proof of Proposition~\ref{prop1}]
Let us focus on a  specific information $x$ in our graph. Suppose we have exactly  $k$ sites knowing $x$. Then, selecting an edge that connects a site that knows $x$ to a site that do not know $x$ will lead to a situation where $k+1$ sites know $x$,
and there $k(n-k)$ edges with this property; otherwise, the number of sites that know~$x$ 
remains unchanged.
So, we change the system from $k$ sites knowing $x$ to $k+1$ sites knowing $x$ with probability $\frac{k(n-k)}{{n \choose 2}}=\frac{2k(n-k)}{n(n-1)}$ at each attempt. Now we use the representation $\tau_x=\sum_{k=1}^{n-1} \Delta_k$, where $\Delta_k$ is the time, with $k$ sites knowing~$x$, for the information be shared with a new site that does not know $x$. Since $\Delta_k \sim Geometric\left(\frac{2k(n-k)}{n(n-1)}\right)$, we write
\begin{align}
\label{decom} M_n(1)&=\ex[\tau_x]=\sum \limits_{k=1}^{n-1} \ex\left[\Delta_k\right]\\
\nonumber &=\sum \limits_{k=1}^{n-1} \frac{n(n-1)}{2k(n-k)}\\
%\nonumber &=\frac{(n-1)}{2}\sum \limits_{k=1}^{n-1} \frac{n}{k(n-k)}\\
\nonumber &=\frac{(n-1)}{2}\sum \limits_{k=1}^{n-1} \left[ \frac{1}{k}+\frac{1}{n-k} \right]\\
\nonumber &=(n-1)H(n-1)\sim n \ln n,
\end{align}
and this concludes the proof.
\end{proof}

Now we present here some lemmas  that we need in order to prove our two main results: Theorems~\ref{teoG} and~\ref{mainteo}.

By \eqref{Mn} and \eqref{An} we can use a simple coupling argument to obtain monotonicity of $M_n$ and $A_n$:
\begin{equation}
\label{monoM}
M_n(1)\leq M_n(2) \leq M_n(3) \leq \ldots \leq M_n(n-1) \leq M_n(n)
\end{equation}
and 
\begin{equation}
\label{monoA}
A_n(2)\leq A_n(3) \leq A_n(4) \leq \ldots \leq A_n(n-1) \leq A_n(n).
\end{equation}

Now we study the situation where each site starts with one different information. First we will consider the complete graph with the quantities~$A_n$ and $M_n$ defined in $\eqref{An}$ and $\eqref{Mn}$. The next lemmas will help us put both~$A_n$ and $M_n$ in order.
\begin{lemma} For $k=1,2,\ldots,n-1$ we have
\[
A_n(k+1)\leq M_n(k)
\]
and for $k=1,2,\ldots,n-2$ we have
\[
M_n(k)\leq A_n(k+2).
\]
\end{lemma}
 
\begin{proof}

Let us consider the complete graph $K_n$ with two simultaneous propagation processes coupled by using the same edges choices. Process 1 and process 2 will have the initial scenarios $\mathcal{I}_0^1$ and $\mathcal{I}_0^2$ respectively. Because of our coupling, we can state that the propagation of $\{2,3,\ldots,k+1\}$ in process~1 implies the propagation of informations $\{1,2,\ldots,k\}$ in process~2. So we have
\[
\tau_{\{2,3,\ldots,k+1\}}^{\mathcal{I}_0^1} \geq \tau_{\{1,2,\ldots,k\}}^{\mathcal{I}_0^2}.
\]
Taking expectations in both sides gives the first result.

 For the second result we just need to observe that since ${\{2,3,\ldots,k+1\}}\subset {\{1,2,\ldots,k+1\}}$, we have
\[
\tau_{\{2,3,\ldots,k+1\}}^{\mathcal{I}_0^2}\leq \tau_{\{1,2,\ldots,k+1\}}^{\mathcal{I}_0^2}
\]
Again taking expectation in both sides we obtain the second result.
\end{proof}

The above lemma implies that
\begin{equation}
\label{order}
A_n(2)\leq M_n(1) \leq A_n(3) \leq M_n(2) \leq \ldots A_n(n) \leq M_n(n-1) \leq M_n(n).
\end{equation}

Now we state our main tool for the proofs of our asymptotic result.
\begin{lemma}
\label{lemmaMA}
For $2 \leq k \leq n-1$ it holds that
\begin{equation*}
\frac{(2n-k-1)k}{2} M_n(k)={n \choose 2} +k(n-k) A_n(k+1)+{k \choose 2}A_n(k).
\end{equation*}

\end{lemma}
\begin{proof} Let us, similarly to the proof of Proposition~\ref{prop1}, analyze the different outcomes in the choice of the first edge of the propagation process. Consider the initial scenario~$\mathcal{I}_0^1$ and let us look only at the first $k$ informations, that is, in the beginning of the process we have $k$ different informations in $k$ distinct sites, and no other informations of interest in any remaining sites. So when we choose one of the ${n \choose 2}$ edges, there are three possible outcomes:
\begin{itemize}
\item We can choose an edge that links two sites with no informations of interest. In this case the configuration of the informations do not alter. We have ${n-k \choose 2}$ edges of this kind. Because the informations remain the same, we have the same expectation $M_n(k)$ of the remaining time as before the choice.
\item We can choose an edge that links a site with a information of interest to a site with no informations. In this case we will have a specific information known by two sites. We have $k(n-k)$ edges of this kind. After the choice, we have a information known by two sites and $k-1$ informations known by just one site, then our remaining time until all sites know all informations is $A_n(k+1)$.
\item We can choose a edge that link two sites with informations of interest. In this case they will share their informations and know both informations. We have ${k \choose 2}$ edges of this kind. The trick here is that since only those two sites will know these informations, they will propagate always together and so we can consider both as just one information, so the expected remaining time until everyone knows all informations is $A_n(k)$.
\end{itemize}
With this we can derive the following relation between $A_n$ and $M_n$:
\begin{align*}
&M_n(k)=1+\frac{{n-k \choose 2}}{{n \choose 2}}M_n(k)+\frac{k(n-k)}{{n \choose 2}}A_n(k+1)+\frac{{k \choose 2}}{{n \choose 2}}A_n(k), 
% \intertext{so}
% &{n \choose 2} M_n(k)={n \choose 2}+{n-k \choose 2}M_n(k) 
% +k(n-k)A_n(k+1)+{k \choose 2}A_n(k), \intertext{therefore}
% &\left[{n \choose 2} -{n-k \choose 2}  \right] M_n(k)
% ={n \choose 2} +k(n-k)A_n(k+1)+{k \choose 2}A_n(k), \intertext{and finally}
\intertext{which can be rewritten as}
&\frac{(2n-k-1)k}{2} M_n(k)={n \choose 2} +k(n-k)A_n(k+1)+{k \choose 2}A_n(k).
\end{align*}
This concludes the  proof of Lemma~\ref{lemmaMA}.
\end{proof}

Lemma \ref{lemmaMA} allows us to associate the quantities~$M_n$ and~$A_n$,
and is a key element for working with $M_n(n)$. 
The last lemma needed in order to prove Theorem~\ref{mainteo} is presented now. The decomposition of the total propagation time used in its proof is also the key to proving Theorem~\ref{teoG}.

\begin{lemma} 
\label{lemmareverse}
Consider the rumor propagation process on a graph ${G}=(V,E)$ with initial configuration $\mathcal{I}_0^1$ and for $x_k$ in $V$ with initial information $k$ define the stopping time
\[
Y_k=\inf\left\{t>0: i_{x_k}(t)=\{1,2,\ldots,|V|\}\right\}.
\]
So $Y_k$ is the time until site $x_k$ knows all informations. Then we have
\[
Y_k\laweq \tau_k,
\]
and then as a consequence, in the complete graph $K_n$:
\begin{equation}
\ex[Y_k]= M_n(1).
\end{equation}

\begin{proof}
Let us denote by $\Gamma_k$ the set of finite sequences of edges $\gamma=(e_1,e_2,\ldots,e_j)$ of arbitrary size with the following property  denoted by \textit{Property~1}: If the propagation process follows this edges order, then information $k$ (initially known only by site $x_k$) becomes propagated among all sites. This property can be rigorously defined as: $\gamma \in \Gamma_k$ if and only if for each $y\neq x_k$ in $V$ there exist a subsequence $\gamma_y$ of~$\gamma$ forming a path from $x_k$ to $y$.

Since we choose our edges uniformly in the process, we can say that the probability of a certain sequence $\gamma \in \Gamma_x$ of length $\ell(\gamma)$ is
\[
\mathbb{P}[\gamma]=\left(\frac{1}{|E|}\right)^{\ell(\gamma)},
\]
which only depends on the length of the sequence and the number of edges. %Now, consider $\gamma=e_1e_2\ldots e_{\ell(\gamma)}$ a sequence of edges in $\Gamma_k$. This sequence propagates the information $k$ from site $x_k$ to other sites. 

We define the property of a sequence $\gamma$ taking all informations to $x_k$ and denote it by \textit{Property 2}: for all $y\neq x_k$ in $V$ there exist a subsequence $\gamma_y^*$ forming a path from $y$ to $x_k$. 

Then, for any $\gamma=e_1e_2\ldots e_{\ell(\gamma)} \in \Gamma_k$, if we revert the order of our choices, i.e., create the sequence~$\gamma'=e_{l(\gamma)}e_{{l(\gamma)}-1}\ldots e_1$, then all paths from $x_k$ to $y$ formed by subsequences becomes paths from $y$ to $x_k$, then, if reverted, a sequence with \textit{Property 1} becomes a sequence with \textit{Property 2}. 

\begin{figure}[h]
\centering
\includegraphics[width=15cm]{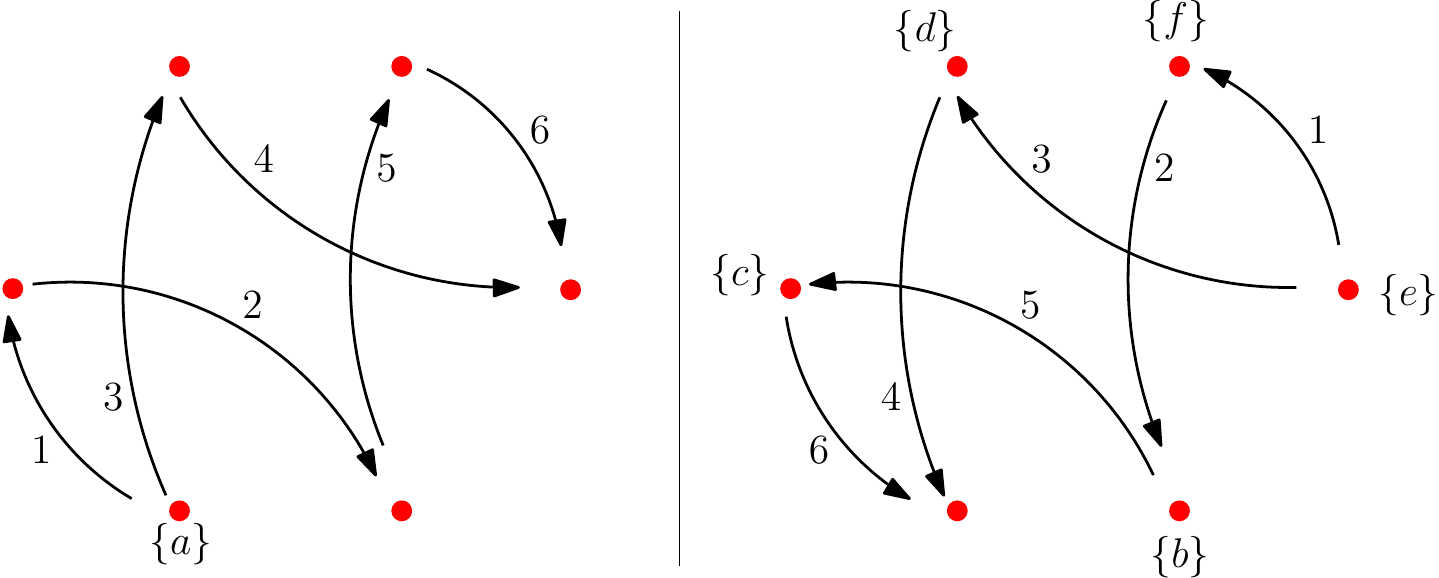}
\caption{In the left a sequence of edges with \textit{Property~1}  and in the right the reverse sequence of edges with \textit{Property~2}. Only the edges choices that really propagates the information are labeled by their choice order.}
\label{fig:pic4}
\end{figure}

As the sequence probability only depends on the sequence size, we can get the following equivalence:
\begin{align*}
\pr [\tau_1 \leq k]&= \sum \limits_{\gamma \in \Gamma, |\gamma|=k} \pr[\gamma]= \sum \limits_{\gamma \in \Gamma, |\gamma|=k}\pr[\gamma']=\pr[Y_x \leq k].
\end{align*}
This concludes the proof of the equality in distribution.
\end{proof}

\end{lemma}

%main
\section{Proofs of the main results}

\subsection{Proof of Theorem~\ref{mainteo}}
We prove Theorem~\ref{mainteo} by obtaining an upper and a lower bound for $M_n(n)$ that have the same asymptotic behavior. 

\begin{lemma}
\label{lower}
 For $n\geq4$ we have
\[
M_n(n) \geq \frac{3}{2}M_n(1)-\frac{3}{4}(n-1).
\]
\end{lemma}
\begin{proof}

 From Lemma~\ref{lemmaMA} we have
\[
\frac{(2n-k-1)k}{2}M_n(k)={n \choose 2} +k(n-k)A_n(k+1)+{k \choose 2}A_n(k).
\]
We can rewrite this as
\begin{align*}
% {n \choose 2} =k(n-k)\left[M_n(k)-A_n(k+1)\right]
% +{k \choose 2}\left[M_n(k)-A_n(k)\right], \intertext{and rewrite again as}
\frac{n(n-1)}{k} =2(n-k) \left[M_n(k)-A_n(k+1)\right]+(k-1)\left[M_n(k)-A_n(k)\right].
\end{align*}
By $\eqref{order}$ we have $A_n(k+1) \geq M_n(k-1)$ and $A_n(k) \geq M_n(k-2)$, then
\begin{align}
\frac{n(n-1)}{k} \leq& 2(n-k) \left[M_n(k)-M_n(k-1)\right]+(k-1)  \left[M_n(k)-M_n(k-2)\right]&\nonumber\\
&=2(n-k) \left[M_n(k)-M_n(k-1)\right]+(k-1) \left[M_n(k)-M_n(k-1)\right]&\nonumber\\
&+(k-1) \left[M_n(k-1)-M_n(k-2)\right]&\nonumber\\
&=(2n-1) \left[M_n(k)-M_n(k-1)\right]-k \left[M_n(k)-M_n(k-1)\right]&\nonumber\\
&\label{eqntele}+(k-1) \left[M_n(k-1)-M_n(k-2)\right].&
\end{align}
Now denote $b_k=k \left[M_n(k)-M_n(k-1)\right]$, so we can rewrite $\eqref{eqntele}$ as
\begin{equation}
\label{eqntele2}
\frac{n(n-1)}{k} \leq (2n-1) \left[M_n(k)-M_n(k-1)\right]-b_k+b_{k-1}.
\end{equation}
Summing $\eqref{eqntele2}$ in $k$ from $3$ to $n-1$ we have:
\begin{align*}
n(n-1)&\Big(H(n-1)-\frac{3}{2}\Big) \\
&\leq (2n-1) \left[M_n(n-1)-M_n(2)\right]-b_{n-1}+b_{2} \\
&= nM_n(n-1)+(n-1)M_n(n-2)-(2n-3)M_n(2)-2M_n(1)\\
&\leq (2n-1)\left[M_n(n)-M_n(1)\right] \\
&\leq  2n\left[M_n(n)-M_n(1)\right].
\end{align*}
But $n(n-1)H(n-1)=nM_n(1)$, using this in the above inequality and dividing by~$n$ we have
\begin{align*}
(n-1)H(n-1)-\frac{3}{2}(n-1)&=M_n(1)-\frac{3}{2}(n-1) \\
&\leq 2\left[M_n(n)-M_n(1)\right],
\end{align*}
{and then}
\[
 M_n(n) \geq \frac{3}{2}M_n(1)-\frac{3}{4}(n-1).
\]
This concludes the  proof of the lower bound.
\end{proof}

Lemma \ref{lower} gives us a lower bound to the ratio~$M_n(n)/M_n(1)$, which shows us that this ratio is asymptotically greater than or equal to~$3/2$. We want to get a asymptotic equality for this ratio, so we need a upper bound for~$M_n(n)/M_n(1)$ that converges to~$3/2$. 

\begin{lemma} We have
\label{upper}
\[
M_n(n) \leq \frac{3}{2}M_n(1).
\]
\end{lemma}

\begin{proof}
We begin by creating a random ordering of the sites of the complete graph depending on the edges sequence of the propagation process with initial configuration $\mathcal{I}_0^1$. To define the ordering we will use the random variable $N(t)=|\{x\in V \mid i_x(t)=I_n\}|$, i.e. the number of sites that know all informations in time $t$. We wanted to construct our ordering $\sigma$ according to the time the sites come to know all informations, but it does not work so easily since typically a lot of pairs of sites discover their last information at the same time; so we would have pairs of sites $x\neq y$ with $\sigma(x)=\sigma(y)$. To avoid this problems when this happens we just toss a coin to uniformly select between them.

 We now define $\sigma$ formally. For this we recall from Lemma~\ref{lemmareverse} that   $Y_x=\inf\{t\geq 1 \mid i_x(t)=I_n\}$ is the first time $x$ knows all informations and consider a sequence $\{B_x\}_{x\in V}$ of i.i.d.\  Bernoulli
random variables with parameter $p=1/2$. Then
\[
L_x=\left\{\begin{tabular}{lll}$0$, & if & $N(Y_x)-N(Y_x-1)=1$,\\ $B_\alpha$, & if & $N(Y_x)-N(Y_x-1)=2$ and $x=\alpha=\inf\{y:N(Y_y)=N(Y_x)\}$,\\
$1-B_\alpha$, & if & $N(Y_x)-N(Y_x-1)=2$ and $x\neq \alpha=\inf\{y:N(Y_y)=N(Y_x)\}$.\end{tabular}\right.
\]
Then, we define
\begin{equation}
\sigma(x)=N(Y_x)-L_x.
\end{equation}

The random variable $L_x$  works as an ordering correction here: it does nothing when $x$ learn his last information alone, but when a pair of sites learn their last informations together, then it subtracts 1 from the rank of one of these sites.

The main reason for this construction is that, by symmetry, each site has the same chance to be in any position. So we have a uniform distribution on the set of possibles permutations of sites, i.e., $\pr[\sigma(x)=k]=\frac{1}{n}$ for any $x\in V$ and $k \in [1,n]\cap \mathbb{N}$.

Now we represent the total propagation time as $\tau_V=Y_x+S_x$, where $S_x$ represents the remaining time until every site learn every information after the time $Y_x$. 
 From
Lemma~\ref{lemmareverse} we know the expectation of $Y_x$, let us then focus on $S_x$. By using our ordering and some fixed $x$ we have (remembering that $\ex[S_x\mid\sigma(x)=n]=0$):
\begin{align}
\nonumber \ex[S_x]&=\sum \limits_{k=1}^n \ex[S_x\mid\sigma(x)=k]\cdot \pr[\sigma(x)=k]\\
\label{Sx} &=\frac{1}{n}\sum \limits_{k=1}^{n-1} \ex[S_x\mid\sigma(x)=k]
\end{align}

Observe also that when $\sigma_x=k$ we have at least $k$ sites that know all informations, so if we remove all informations from all other sites except those, we would have~$k$ sites with all informations and each other site with no information. As all informations now propagate together as one, we can couple this situation to the propagation process of one information at a time when $k$ sites know a particular information. 

Since we erased some informations from some sites but did not reduce the number of informations (because the $k$ unaltered sites have all informations), it will take longer for them to propagate and so we can majorize  $\ex[S_x\mid\sigma_x=k]$ by the expected time in this altered situation and using the decomposition $\eqref{decom}$:
\begin{align*}
\ex[S_x\mid\sigma(x)=k] \leq& \sum \limits_{j=k}^{n-1} \ex\left[\Delta_j\right]\\
&=\frac{n-1}{2}\sum \limits_{j=k}^{n-1} \left[\frac{1}{j}+\frac{1}{n-j}\right]\\
&= \frac{n-1}{2}\left(H(n-1)-H(k-1)+H(n-k)\right)\\
&=\frac{M_n(1)}{2} + \frac{n-
1}{2}\left(H(n-k) -H(k-1) \right).
\end{align*}
Then, turning back to \eqref{Sx}, and using the fact that $H(0)=0$ we get
\begin{align*}
\ex[S_x] \leq& \frac{1}{n} \sum \limits_{k=1}^{n-1} \left[ \frac{M_n(1)}{2}+ \frac{n-1}{2}\left(H(n-k) -H(k-1) \right) \right]\\
&= \frac{(n-1)M_n(1)}{2n} +\frac{n-1}{2n} \Big(\sum \limits_{k=1}^{n-1} H(k) - \sum \limits_{k=1}^{n-1} H(k-1) \Big)\\
&= \frac{(n-1)M_n(1)}{2n} +\frac{n-1}{2n}H(n-1)\\
&=\frac{M_n(1)}{2}.
\end{align*}

Finally, together with Lemma~\ref{lemmareverse} we get:
\begin{align*}
M_n(n)=&\ex F_x+\ex S_x\\
\leq& M_n(1)+\frac{M_n(1)}{2}=\frac{3}{2}M_n(1),
\end{align*}
 and so we have the upper bound.
\end{proof}

Now we can conclude the proof of the main result. It comes easily by the previously stated lemmas. 

\begin{proof}[Proof of Theorem~\ref{mainteo}]

Lemmas~\ref{lower} and \ref{upper} give that
\[
\frac{3}{2}M_n(1)-\frac{3}{4}(n-1) \leq M_n(n) \leq \frac{3}{2}M_n(1).
\]
Dividing both sides by $M_n(1)$ we have

\[
\frac{3}{2}-\frac{3(n-1)}{4M_n(1)} \leq \frac{M_n(n)}{M_n(1)} \leq \frac{3}{2},
\]
and by Proposition~\ref{prop1}
\[
\frac{3(n-1)}{4M_n(1)}=\frac{3}{4H(n-1)}\sim \frac{3}{4 \ln n} \to 0, \quad \mbox{as $n \to \infty$},
\]
which concludes the proof.
\end{proof}

\subsection{Proof of Theorem~\ref{teoG}}

\begin{proof}[Proof of Theorem~\ref{teoG}]
Similarly to the proof of Lemma \ref{upper}, we can represent the total propagation time as $\tau_V=Y_x+S_x$, where $Y_x=\inf\{t\geq 1 \mid i_x(t)=I_n\}$ is the first time when $x$ knows all informations, and $S_x$ represents the remaining time after $t_x$ until every other site knows all informations. By Lemma \ref{lemmareverse} we have $\ex[t_x]=\ex[\tau_x]$ .

 Now observe that after time $Y_x$ the site $x$ knows all informations (together with another site, but we will not use this site), so if we consider all these informations together as one and ignore other informations in other sites, the time it takes for the propagation of them together is the same in law as $\tau_x$. As these alterations can only delay more the propagation, then if we couple the original propagation to this alteration we have 
\[
\ex[S_x] \leq \ex[\tau_x];
\] 
taking expectations we obtain
\begin{align*}
\ex[\tau_V]&=\ex[F_x]+\ex[S_x]\\
&\leq\ex[\tau_x]+\ex[\tau_x]\\
&=2\ex[\tau_x],
\end{align*}
which concludes the proof.
\end{proof}

Theorem~\ref{teoG} together with Proposition~\ref{propS} shows that we have a tight upper bound of $2$ for the propagation ratio of general graphs.

\bibliography{refe}

\begin{thebibliography}{1}

\bibitem{FPP}
First-passage percolation, subadditive processes, stochastic networks, and
  generalized renewal theory.

\bibitem{DK}
{Daley, D.J.}; {Kendall, D.G.}
\newblock Stochastic rumours.
\newblock {\em J. Inst. Maths. Applicns. 1}, page 42–55, 1965.

\bibitem{art2}
{Lebensztayn, E.}
\newblock A large deviations principle for the {Maki}-{Thompson} rumour model.
\newblock {\em arXiv:1411.5614}, 2014.

\bibitem{art1}
{Kurtz, T. G.}; {Lebensztayn, E.};~{Leichsenring, A. R.}; {Machado, F. P.}
\newblock Limit theorems for an epidemic model on the complete graph.
\newblock {\em Alea}, 4:45--55, 2008.

\bibitem{art3}
{Pittel, B.}
\newblock On a {Daley}-{Kendall} model of random rumours.
\newblock {\em Journal of Applied Probability}, pages 14--27, 1990.

\bibitem{art4}
{Sudbury, A.}
\newblock The proportion of the population never hearing a rumour.
\newblock {\em Journal of applied probability}, pages 443--446, 1985.

\bibitem{MT}
{Maki, D.P.}; {Thompson, M.}
\newblock {\em Mathematical Models and Applications}.
\newblock Prentice-Hall, Englewood Cliffs, 1973.

\bibitem{flooding}
{van der Hofstad, R.};~{Hooghiemstra, G.}; {van Mieghem, P.}
\newblock The flooding time in random graphs.
\newblock {\em Extremes}, 5(2):111--129, 2002.

\bibitem{art5}
{Watson, R.}
\newblock On the size of a rumour.
\newblock {\em Stochastic processes and their applications}, 27:141--149, 1987.

\end{thebibliography}

\end{document}